\documentclass{amsart}
\usepackage{amssymb} 
\newtheorem{thm}{Theorem}[section]
\newtheorem{cor}[thm]{Corollary}

\theoremstyle{question}

\numberwithin{equation}{section} 

\begin{document}
\author[Abdollahi and Khosravi]{A. Abdollahi \& H. Khosravi}
\title[Right $4$-Engel elements]
{Right $4$-Engel elements of a group}
\address{Department of Mathematics, University of Isfahan, Isfahan 81746-73441, Iran}%
\email{a.abdollahi@math.ui.ac.ir \;\; alireza\_abdollahi@yahoo.com}
\email{hassan\_khosravy@yahoo.com}
\subjclass[2000]{20D45}
\keywords{Right $4$-Engel elements of a group; $4$-Engel groups}
\thanks{}
\begin{abstract}
 We prove that the set of right $4$-Engel elements of a group $G$  is a subgroup for locally nilpotent groups $G$ without
elements of orders $2$, $3$ or $5$; and in this case the normal
closure $\langle x \rangle^G$ is nilpotent of class at most $7$
for each right $4$-Engel elements $x$ of $G$.
\end{abstract}
\maketitle
\section{\bf Introduction and Results}\label{sec1}
 Let $G$ be any group and
$n$  a non-negative integer. For any two
 elements $a$ and $b$ of $G$, we define
inductively $[a,_n b]$ the $n$-Engel commutator of the pair
$(a,b)$, as follows:
$$[a,_0 b]:=a,~ [a,b]=[a,_1 b]:=a^{-1}b^{-1}ab \mbox{ and } [a,_n
b]=[[a,_{n-1} b],b]\mbox{ for all }n>0.$$  An element $x$ of $G$
is called right  $n$-Engel if $[x,_ng]=1$ for all $g\in G$. We denote by $R_n(G)$  the set of all right  $n$-Engel elements of
$G$.  A group $G$ is called $n$-Engel if  $G=R_n(G)$. It is clear that $R_1(G)=Z(G)$ is the center of $G$ and  W.P. Kappe
\cite{kappew}  proved $R_2(G)$ is a characteristic
subgroup of $G$.  Macdonald \cite{macd} has shown
that the inverse or square of a right 3-Engel element need not be
right 3-Engel. Nickel \cite{nick2} generalized Macdonald's result to all $n
\geq 3$. Although Macdonald's example shows that $R_3(G)$ is not
in general a subgroup of $G$, Heineken \cite{hein} has already shown that
if $A$ is the subset of a group $G$ consisting of all elements $a$
such that $a^{\pm 1}\in R_3(G)$, then $A$ is a subgroup if either
$G$ has no element of order $2$ or $A$ consists only of elements
having finite odd order. Newell \cite{newell} proved that the
normal closure of every right $3$-Engel element is nilpotent of
class at most $3$. In Section 2 we prove that if $G$ is a
$2'$-group, then $R_3(G)$ is a subgroup of $G$. Nickel's example
shows that the set of right 4-Engel elements is not a subgroup in
general (see also the first example in Section 4 of \cite{ab1}). In
Section 3, we prove that if $G$ is a locally nilpotent
$\{2,3,5\}'$-group, then $R_4(G)$ is a subgroup of $G$.

Traustason \cite{traus} proved that any locally nilpotent 4-Engel
group $H$ is Fitting of degree at most $4$. This means that the
normal closure of every element of $H$ is nilpotent of class at
most 4. More precisely he proved that if $H$ has no element of
order $2$ or $5$, then $H$ has Fitting degree at most $3$. Now by a result of Havas and Vaughan-Lee \cite{havas}, one knows  any 4-Engel group is  locally nilpotent and so Traustason's  result is true for all $4$-Engel groups. In Section 3, by another result of Traustason \cite{traus2} we show that the normal closure of every right $4$-Engel element in a locally nilpotent
$\{2,3,5\}'$-group, is nilpotent of class at most $7$.

 Throughout the paper we have frequently used {\sf nq}
package of Werner Nickel \cite{Nick3} which is implemented in {\sf GAP} \cite{gap}.
 All given  timings  were obtained on an Intel
Pentium 4-1.70GHz processor with 512 MB running Red Hat Enterprise Linux 5.
\section{\bf Right 3-Engel elements}\label{se2}
Throughout, for any positive integer $k$ and any group $H$,  $\gamma_k(H)$ denotes the $k$th term of the lower central series of $H$.
The main result of  this section implies that  $R_3(G)$ is a subgroup of $G$ whenever $G$ is a $2'$-group.
Newell \cite{newell} proved that
\begin{thm}\label{th1}
Let $G=\langle a,b,c\rangle$ be a group such that $a,b\in R_3(G)$.
Then
\begin{enumerate}
    \item $\langle a,c\rangle$ is nilpotent of class at most $5$ and
    $\gamma_5(\langle a,c\rangle)$ has exponent $2$.
    \item $G$ is nilpotent of class at most $6$.
    \item $\gamma_5(G)/\gamma_6(G)$ has exponent $10$.
    Furthermore $[a,c,b,c,c]^2\in \gamma_6(G)$.
    \item $\gamma_6(G)$ has exponent $2$.
\end{enumerate}
\end{thm}
\begin{thm}\label{th0}
Let $G$ be a group such that $\gamma_5(G)$ has no element of order $2$. Then $R_3(G)$ is a subgroup of $G$.
\end{thm}
\begin{proof}
Let $a,b\in R_3(G)$ and let $c$ be an arbitrary element of $G$.
Thus
\begin{enumerate}
\item[(1)] $[a,c,c,c]=1$.
\item[(2)] $[b,c,c,c]=1$.
\end{enumerate}
Since by our assumption  $\gamma_5(G)$ has no element of order $2$, it follows from Theorem \ref{th1} parts (1), (3) and  (4), respectively that
\begin{enumerate}
\item[(3)] the subgroup $\langle a,c\rangle$
is nilpotent of class at most $4$.
\item[(4)] $[a,c,b,c,c]=1$.
\item[(5)]  the subgroup $\langle a,b,c \rangle$ is nilpotent of
class at most $5$.
\end{enumerate}
To prove $R_3(G)$ is a subgroup, we have to show that both $a^{-1}$ and $ab$ belong to $R_3(G)$.
We first prove  that $a^{-1}\in R_3(G)$. It easily follows from (1) and (3) that:
$$[a^{-1},c,c,c]=[a,c,c,c]^{-1}=1.$$
Therefore  $a^{-1}\in R_3(G)$.\\
We now show that $ab\in R_3(G)$.
\begin{eqnarray}
[ab,c,c,c]&=&\big[[a,c][a,c,b][b,c],c,c\big]\nonumber\\
&=& \big[ [a,c,c]^{[a,c,b][b,c]}\big[[a,c,b][b,c],c\big],c\big] \nonumber\\
&=& \big[ [a,c,c]^{[b,c]}[a,c,b,c]^{[b,c]}[b,c,c], c\big] \nonumber \;\; {\rm by} \;\; (5)\\
&=& \big[ [a,c,c]\big[a,c,c,[b,c]\big][a,c,b,c][b,c,c],c \big] \nonumber \;\; {\rm by} \;\; (5) \\
&=& \big[ [a,c,b,c][b,c,c],c \big] \nonumber \;\; {\rm by} \;\; (1) \;{\rm and}\; (5) \\
&=&  [a,c,b,c,c]^{[b,c,c]} \nonumber \;\; {\rm by} \;\; (2)\\
&=& 1           \nonumber \;\; {\rm by} \;\; (4)
\end{eqnarray}
This completes the proof.
\end{proof}
Now we give a proof of Theorem \ref{th0} by using {\sf nq}
package of Werner Nickel \cite{Nick3} which is implemented in {\sf GAP} \cite{gap}. Note that  the knowledge of   Theorem \ref{th1} is crucial in the following proof. The package  {\sf nq} has the capability of computing the largest nilpotent quotient (if it exists) of a finitely generated group with finitely many identical relations and finitely many relations. For example, if we want to construct the largest nilpotent quotient of a group $G$ as follows
$$\langle x_1,\dots,x_n \;|\; r_1(x_1,\dots,x_n)=\dots=r_m(x_1,\dots,x_n)=1, w(x_1,\dots,x_n,y_1,\dots,y_k)=1\rangle,$$
where $r_1,\dots,r_m$ are relations on $x_1,\dots,x_n$ and $w(x_1,\dots,x_n,y_1,\dots,y_k)=1$ is an identical relation  in the group $\langle x_1,\dots,x_n\rangle$, one may apply the following  code to use the package {\sf nq} in {\sf GAP}:
\begin{verbatim}
LoadPackage("nq"); #nq package of Werner Nickel #
F:=FreeGroup(n+k);
L:=F/[r1(F.1,...,F.n),...,rm(F.1,...,F.n),w(F.1,...,F.n,F.(n+1),...,F.(n+k))];
H:=NilpotentQuotient(L,[F.(n+1),...,F.(n+k)]);
\end{verbatim}
Note that we need to construct the free group of rank $n+k$ because  as well as the $n$ generators for $G$ we also have an
identical relation with $k$ free variables. \\
 Note that the function {\sf NilpotentQuotient(L)} attempts to compute the largest nilpotent quotient of {\sf L} and it will terminate only if {\sf L} has a largest nilpotent quotient. \\

\noindent{\bf Second Proof of Theorem \ref{th0}.} By Theorem
\ref{th1}, we know that $\langle x,y,z\rangle$ is nilpotent if
$x,y\in R_3(G)$ and $z\in G$. We now construct the largest
nilpotent group $H=\langle a,b,c \rangle$ such that $a,b\in
R_3(H)$ and $c\in H$, by {\sf nq} package.
\begin{verbatim}
 LoadPackage("nq");
 F:=FreeGroup(4);a1:=F.1; b1:=F.2; c1:=F.3; x:=F.4;
 L:=F/[LeftNormedComm([a1,x,x,x]),LeftNormedComm([b1,x,x,x])];
 H:=NilpotentQuotient(L,[x]);
 a:=H.1; b:=H.2; c:=H.3;  d:=LeftNormedComm([a^{-1},c,c,c]);
 e:=LeftNormedComm([a*b,c,c,c]);  Order(d); Order(e);
 C:=LowerCentralSeries(H);  d in C[5]; e in C[5];
\end{verbatim}
Then if we consider the elements $d=[a^{-1},c,c,c]$ and
$e=[ab,c,c,c]$ of $H$, we can see by above command in {\sf GAP} that $d$ and $e$ are elements of $\gamma_5(H)$ and have orders $2$ and
$4$, respectively. So, in the group $G$, we have   $d=e=1$.
This completes the proof. $\hfill \Box$

Note that, the second proof of Theorem \ref{th0} also shows the necessity of assuming  that $\gamma_5(G)$ has no element of order $2$.
\section{\bf Right 4-Engel elements}\label{se3}
Our main result in this section is to prove the following.
\begin{thm}\label{th2}
Let $G$ be a  $\{2,3,5\}'$-group such that $\langle a,b,x\rangle$ is nilpotent for all $a,b\in R_4(G)$ and any $x\in G$. Then $R_4(G)$
is a subgroup of $G$.
\end{thm}
\begin{proof}
Consider the `freest' group, denoted by $U$, generated by two elements $u$,$v$
with $u$ a right 4-Engel element. We mean this by the group $U$
given by the presentation
$$\langle u,v \;|\; [u,_4 x]=1 \;\;\text{for all words}\;\; x \in F_2\rangle,$$
where $F_2$ is the free group generated by $u$ and $v$.
We do not know whether $U$ is nilpotent or not.
Using the {\sf nq} package shows that the group
$U$ has a largest nilpotent quotient $M$ with
class $8$.
By the following code, the
group $M$ generated by a right $4$-Engel element $a$ and an
arbitrary element $c$ is constructed.
We then see that the element $[a^{-1},c,c,c,c]$ of $M$
is of order  $375=3\times 5^3$.  Therefore the inverse of a right
$4$-Engel element of $G$ is again a right $4$-Engel element. The
following code in {\sf GAP} gives a proof of the latter claim. The computation
was completed in about 24 seconds.
\begin{verbatim}
 F:=FreeGroup(3); a1:=F.1; b1:=F.2; x:=F.3;
 U:=F/[LeftNormedComm([a1,x,x,x,x])];
 M:=NilpotentQuotient(U,[x]);
 a:=M.1; c:=M.2;
 h:=LeftNormedComm([a^-1,c,c,c,c]);
 Order(h);
\end{verbatim}
We now show that the product of every two
right 4-Engel elements in $G$
is a right 4-Engel element. Let $a,b\in R_4(G)$ and $c\in G$. Then
we claim that
$$H=\langle a,b,c\rangle \;\; \text{is nilpotent of class at most}\; 7. \;\;\;(*)$$ By induction on the nilpotency class of $H$, we may assume that $H$ is
nilpotent of class at most 8. Now we construct the largest nilpotent group
$K=\langle a_1,b_1,c_1\rangle$ of class 8 such that $a_1,b_1\in R_4(K)$.
\begin{verbatim}
 F:=FreeGroup(4);A:=F.1; B:=F.2; C:=F.3; x:=F.4;
 W:=F/[LeftNormedComm([A,x,x,x,x]),LeftNormedComm([B,x,x,x,x])];
 K:=NilpotentQuotient(W,[x],8);
 LowerCentralSeries(K);
\end{verbatim}
The computation took about 22.7 hours. We see that $\gamma_8(K)$
has exponent $60$. Therefore, as $H$ is a  $\{2,3,5\}'$-group, we have
$\gamma_8(H)=1$ and this completes the proof of our  claim $(*)$. \\
Therefore we have proved that any nilpotent group without elements of orders $2$, $3$ or $5$ which is generated by three elements two of which are right $4$-Engel, is nilpotent of class at most $7$.\\
Now we construct, by the {\sf nq} package, the largest nilpotent group $S$ of class $7$ generated by two right $4$-Engel elements $s,t$ and an arbitrary element $g$.  Then one can find  by {\sf GAP} that the order of   $[st,g,g,g,g]$ in $S$  is
  300. Since $H$ is a  quotient of $S$, we have that $[ab,c,c,c,c]$ is of order dividing $300$ and so it is trivial, since $H$ is a $\{2,3,5\}'$-group.
 This completes the proof.
\end{proof}
\begin{cor}\label{co1}
Let $G$ be a  $\{2,3,5\}'$-group such that $\langle a,b,x\rangle$ is nilpotent for all $a,b\in R_4(G)$ and for any $x\in G$. Then $R_4(G)$ is a nilpotent group of class at most $7$. In particular,  the normal closure of every right $4$-Engel element of group $G$ is nilpotent
of class at most $7$.
\end{cor}
\begin{proof}
By  Theorem \ref{th2}, $R_4(G)$ is a subgroup of $G$ and so it
is a 4-Engel group. In \cite{traus2} it is shown that every
locally nilpotent 4-Engel $\{2,3,5\}'$-group is nilpotent of class at most 7.
Therefore $R_4(G)$ is nilpotent of class at most 7. Since $R_4(G)$ is a normal set, the second part follows easily.
\end{proof}
Therefore, to prove that the normal closure of any right $4$-Engel element of a $\{2,3,5\}'$-group $G$ is nilpotent, it is enough to show that
$\langle a,b,x\rangle$ is nilpotent for all $a,b\in R_4(G)$ and for any $x\in G$.
\begin{cor}
In any $\{2,3,5\}'$-group, the normal closure of any right $4$-Engel element is nilpotent if and only if every $3$-generator subgroup in which two of the generators can be chosen to be  right $4$-Engel, is nilpotent.
\end{cor}
\begin{proof}
By Corollary \ref{co1}, it is enough to show that  a $\{2,3,5\}'$-group $H=\langle a,b,x\rangle$ is nilpotent whenever $a,b\in R_4(H)$, $x\in H$ and both $\langle a\rangle^H$ and $\langle b\rangle ^H$ are nilpotent. Consider the subgroup $K=\langle a\rangle ^H\langle b\rangle^H$ which is nilpotent by Fitting's theorem. Now we prove that $K$ is finitely generated. We have $K=\langle a,b\rangle^{\langle x\rangle}$ and since $a$ and $b$ are both right $4$-Engel, it is well-known that
$$\langle a\rangle^{\langle x\rangle}=\langle a,a^x,a^{x^2},a^{x^3}\rangle \;\;\text{and}\;\; \langle b\rangle^{\langle x\rangle}=\langle b,b^x,b^{x^2},b^{x^3}\rangle,$$
and so $$K=\langle a,a^x,a^{x^2},a^{x^3},b,b^x,b^{x^2},b^{x^3} \rangle.$$
It follows that $H$  satisfies maximal condition on its subgroups as it is (finitely generated nilpotent)-by-cyclic. Now by a famous result of Baer \cite{Baer} we have that $a$ and $b$ lie in the $(m+1)$th term $\zeta_m(H)$ of the upper central series of $H$ for some positive integer $m$. Hence $H/\zeta_m(H)$ is cyclic and so $H$ is nilpotent. This completes the proof.
\end{proof}
We conclude this section with the following interesting information on the group $M$ in the proof of Theorem \ref{th2}.
In fact, for the largest nilpotent group $M=\langle a,b\rangle$ relative to $a\in R_4(M)$, we have that $M/T$ is isomorphic to the largest (nilpotent) $2$-generated $4$-Engel group $E(2,4)$, where $T$ is the torsion subgroup of $M$ which is a $\{2,3,5\}$-group. Therefore, in a nilpotent $\{2,3,5\}'$-group, a right $4$-Engel element with an arbitrary element generate a $4$-Engel group. This can be seen by  comparing the presentations of $M/T$ and $E(2,4)$ as follows. One can obtain two finitely presented groups {\sf G1} and {\sf G2} isomorphic to $M/T$ and $E(2,4)$, respectively by {\sf GAP}:
\begin{verbatim}
MoverT:=FactorGroup(M,TorsionSubgroup(M));
E24:=NilpotentEngelQuotient(FreeGroup(2),4);
iso1:=IsomorphismFpGroup(MoverT);iso2:=IsomorphismFpGroup(E24);
G1:=Image(iso1);G2:=Image(iso2);
\end{verbatim}
Next, we find  the relators  of the groups {\sf G1} and {\sf G2} which are two sets of relators on 13 generators by the following command in {\sf GAP}.
\begin{verbatim}
r1:=RelatorsOfFpGroup(G1);r2:=RelatorsOfFpGroup(G2);
\end{verbatim}
Now, save these two sets of relators  by {\sf LogTo} command of {\sf GAP} in a file and go to the file to delete the terms as
\begin{verbatim}
<identity ...>
\end{verbatim}
in the sets {\sf r1} and {\sf r2}. Now call these two modified sets {\sf R1} and {\sf R2}. We show that {\sf R1=R2} as two sets of elements of the  free group {\sf f}  on 13 generators {\sf f1,f2,...,f13}.
\begin{verbatim}
f:=FreeGroup(13);
f1:=f.1;f2:=f.2;f3:=f.3;f4:=f.4;f5:=f.5;f6:=f.6;
f7:=f.7;f8:=f.8;f9:=f.9;f10:=f.11;f12:=f.12;f13:=f.13;
\end{verbatim}
Now by {\sf Read} function, load the file in {\sf GAP} and type the  simple command
{\sf R1=R2}. This gives us {\sf true} which shows $G_1$ and $G_2$ are two finitely presented groups with the same relators and generators and so they are isomorphic. We do not know if there is a guarantee that if someone else does as we did, then he/she finds the same relators for {\sf Fp} groups {\sf G1} and {\sf G2}, as we have found. Also we remark that using function {\sf IsomorphismGroups} to test if $G_1\cong G_2$, did not give us a result in less than 10 hours and we do not know whether this function can give us a result or not.

We summarize the above discussion as following.
\begin{thm}
Let $G$ be a nilpotent group generated by two elements, one of which is  a right $4$-Engel element. If $G$ has no element of order $2$, $3$ or $5$, then $G$ is a $4$-Engel group of class at most $6$.
\end{thm}
\noindent{\bf Acknowledgments.} The authors are grateful to the referee for his/her careful reading and insightful  comments.
The research of the first author is financially supported by the Center of Excellence for Mathematics, University of Isfahan.

\end{document}